\documentclass[12pt]{article}
\usepackage[utf8]{inputenc}
\usepackage[T1]{fontenc}
\usepackage[top=25mm, bottom=25mm, left=25mm, right=25mm, marginparwidth=28mm]{geometry}
\usepackage{amsfonts,amsmath,amsthm,amssymb}
\usepackage{hyperref,enumitem,enumerate}
\usepackage{xcolor}
\hypersetup{
    urlcolor=black,
}
\usepackage{graphics}
\usepackage{pstricks,pst-node,pst-arrow}
\PassOptionsToPackage{usenames,dvipsnames,svgnames}{xcolor}  
\usepackage{tikz}
\usetikzlibrary{arrows,positioning,automata, fit}
\usetikzlibrary{arrows.meta}
\usepackage{float}
\usetikzlibrary{graphs,graphs.standard,quotes}
\usetikzlibrary{matrix, positioning, fit}

\usepackage{xifthen}

\usepackage[normalem]{ulem}
\usepackage{xcolor}
\newcommand\redsout{\bgroup\markoverwith{\textcolor{magenta}{\rule[0.5ex]{4pt}{0.8pt}}}\ULon}
\newcommand\bluesout{\bgroup\markoverwith{\textcolor{cyan}{\rule[0.5ex]{4pt}{0.8pt}}}\ULon}
\newcommand\greensout{\bgroup\markoverwith{\textcolor{green}{\rule[0.5ex]{4pt}{0.8pt}}}\ULon}
\newcommand\orangesout{\bgroup\markoverwith{\textcolor{orange}{\rule[0.5ex]{4pt}{0.8pt}}}\ULon}

\def\ArgOne#1{\arrowcolor#1,\relax}
\def\arrowcolor#1,#2,#3\relax{#1}

\def\ArgTwo#1{\arrownumber#1,\relax}
\def\arrownumber#1,#2,#3\relax{#2}

\def\ArrowColor#1{%
    \ArgOne{#1}%
}%

\def\ArrowNumber#1{%
    \ArgTwo{#1}%
}%

\tikzset{pics/.cd,
  pic InnerGraph/.style 2 args={code={
    \def\innerradius{1.5}
    \def\radius{3}
    \def\bendangle{5}

    \foreach \v/\lab in {1, 2, 3, 4, 5, 6, 7, 8}{  
      \node[inner sep=0pt,minimum size=12pt] (n1\v) at (45-\v*45:\innerradius cm) {};
    } 
    \foreach \v/\vnext in {1/4, 8/5}{
      \draw [#1, -{Stealth[length=3mm, width=2mm]}] (n1\v) to[bend right=\bendangle] (n1\vnext); 
      \draw [#1, -{Stealth[length=3mm, width=2mm]}] (n1\vnext) to[bend left=\bendangle] (n1\v); 
    }
    \foreach \v/\vnext in {2/7, 3/6}{
      \draw [#2, -{Stealth[length=3mm, width=2mm]}] (n1\v) to[bend right=\bendangle] (n1\vnext); 
      \draw [#2, -{Stealth[length=3mm, width=2mm]}] (n1\vnext) to[bend left=\bendangle] (n1\v); 
    }
    \node[circle,draw=black!80, line width=0.5mm, inner sep=0pt, fit={(n11) (n12) (n13) (n14) (n15) (n16) (n17) (n18)}] (n1) {};

  }}
}

\tikzset{pics/.cd,
  pic Eliplse/.style n args={7}{code={
  
    \def\innerradius{1.5}
    \def\radius{3}
    \def\innernodesize{0.15}

    \newcommand{\pairarrows}[5]{
    \def\bendangle{##5}
      \begin{tikzpicture}[overlay]
        \ifthenelse{##4 = 1}{\draw [##1, -{Stealth[length=3mm, width=2mm]}] (##2) to[bend right=\bendangle] (##3)}{};
        \ifthenelse{##4 = 2}{\draw [##1, -{Stealth[length=3mm, width=2mm]}] (##3) to[bend left=\bendangle] (##2)}{};
        \ifthenelse{##4 = 3}{
          \draw [##1, -{Stealth[length=3mm, width=2mm]}] (##2) to[bend right=\bendangle] (##3);
          \draw [##1, -{Stealth[length=3mm, width=2mm]}] (##3) to[bend left=\bendangle] (##2);}{};
      \end{tikzpicture}  
    }

    \node[circle, minimum size=\innernodesize cm, inner sep=0pt, outer sep=0pt] (ne1) {};
    \node[circle, minimum size=\innernodesize cm, inner sep=0pt, outer sep=0pt]  (ne2) [right = 2cm of ne1] {};
    \node[circle, minimum size=\innernodesize cm, inner sep=0pt, outer sep=0pt]  (ne3) [below =0.5 cm of ne1] {};
    \node[circle, minimum size=\innernodesize cm, inner sep=0pt, outer sep=0pt, label={[label distance=1cm]90:#7}]  (ne4) [below =0.5 cm of ne2] {};
    
    \pairarrows{\ArrowColor{#1}}{ne1.south}{ne2.south}{\ArrowNumber{#1}}{8}
    \pairarrows{\ArrowColor{#2}}{ne1.east}{ne2.west}{\ArrowNumber{#2}}{0}
    \pairarrows{\ArrowColor{#3}}{ne1.north}{ne2.north}{\ArrowNumber{#3}}{-8}
    
    \pairarrows{\ArrowColor{#4}}{ne3.south}{ne4.south}{\ArrowNumber{#4}}{8}
    \pairarrows{\ArrowColor{#5}}{ne3.east}{ne4.west}{\ArrowNumber{#5}}{0}
    \pairarrows{\ArrowColor{#6}}{ne3.north}{ne4.north}{\ArrowNumber{#6}}{-8}

    \node[rectangle, rounded corners=0.5cm, draw=black!80, line width=0.5mm, inner sep=10pt, fit={(ne1) (ne2) (ne3) (ne4)}] (n1) {};

  }}
}

\newcommand{\outerpairarrows}[5]{
    \def\bendangle{#5}
      \begin{tikzpicture}[overlay]
        \ifthenelse{#4 = 1}{\draw [#1, -{Stealth[length=3mm, width=2mm]}] (#2) to[bend right=\bendangle] (#3)}{};
        \ifthenelse{#4 = 2}{\draw [#1, -{Stealth[length=3mm, width=2mm]}] (#3) to[bend left=\bendangle] (#2)}{};
        \ifthenelse{#4 = 3}{
          \draw [#1, -{Stealth[length=3mm, width=2mm]}] (#2) to[bend right=\bendangle] (#3);
          \draw [#1, -{Stealth[length=3mm, width=2mm]}] (#3) to[bend left=\bendangle] (#2);}{};
      \end{tikzpicture}  
    }

\tikzset{pics/.cd,
  pic InnerGraph/.style 2 args={code={
    \def\innerradius{1}
    \def\radius{3}
    \def\bendangle{5}

    \foreach \v/\lab in {1, 2, 3, 4, 5, 6, 7, 8}{  
      \node[inner sep=0pt,minimum size=0pt] (n1\v) at (60-\v*45: \innerradius cm) {};
    } 
    \foreach \v/\vnext in {1/4, 8/5}{
      \draw [#1, -{Stealth[length=3mm, width=2mm]}] (n1\v) to[bend right=\bendangle] (n1\vnext); 
      \draw [#1, -{Stealth[length=3mm, width=2mm]}] (n1\vnext) to[bend left=\bendangle] (n1\v); 
    }
    \foreach \v/\vnext in {7/2, 6/3}{
      \draw [#2, -{Stealth[length=3mm, width=2mm]}] (n1\v) to[bend right=\bendangle] (n1\vnext); 
      \draw [#2, -{Stealth[length=3mm, width=2mm]}] (n1\vnext) to[bend left=\bendangle] (n1\v); 
    }
    \node[rectangle, rounded corners=0.8cm, draw=black!80, line width=0.5mm, inner sep=0pt, fit={(n11) (n12) (n13) (n14) (n15) (n16) (n17) (n18)}] (n1) {};

  }}
}

\tikzset{pics/.cd,
  pic InnerGraphBig/.style 2 args={code={
    \def\innerradius{1}

    \def\radius{3}
    \def\bendangle{5}

    \node[] (n11) at (-1.5,1) {};
    \node[] (n12) [right of=n11] {};
    \node[] (n13) [right of=n12] {};
    \node[] (n14) [right of=n13] {};
    \node[] (n15) [below of=n11] {};
    \node[] (n16) [below of=n15] {};
    \node[] (n17) [right of=n16] {};
    \node[] (n18) [right of=n17] {};
    \node[] (n19) [right of=n18] {};
    \node[] (n110) [above of=n19] {};

    \foreach \v/\vnext in {2/5, 3/6, 4/7, 8/10}{
      \draw [#1, -{Stealth[length=3mm, width=2mm]}] (n1\v) to[bend right=\bendangle] (n1\vnext); 
      \draw [#1, -{Stealth[length=3mm, width=2mm]}] (n1\vnext) to[bend left=\bendangle] (n1\v); 
    }
    \foreach \v/\vnext in {5/7, 1/8, 2/9, 3/10}{
      \draw [#2, -{Stealth[length=3mm, width=2mm]}] (n1\v) to[bend right=\bendangle] (n1\vnext); 
      \draw [#2, -{Stealth[length=3mm, width=2mm]}] (n1\vnext) to[bend left=\bendangle] (n1\v); 
    }
    \node[rectangle, rounded corners=0.8cm, draw=black!80, line width=0.5mm, inner sep=0pt, fit={(n11) (n12) (n13) (n14) (n15) (n16) (n17) (n18)}] (n1) {};

  }}
}

\tikzset{pics/.cd,
  pic InnerGraphGIGANT/.style 2 args={code={
    \def\innerradius{1}

    \def\radius{3}
    \def\bendangle{5}

    \node[] (n11) at (-1.5,1) {};
    \node[] (n12) [right of=n11] {};
    \node[] (n13) [right of=n12] {};
    \node[] (n14) [right of=n13] {};
    \node[] (n15) [below of=n11] {};
    \node[] (n16) [below of=n15] {};
    \node[] (n17) [right of=n16] {};
    \node[] (n18) [right of=n17] {};
    \node[] (n19) [right of=n18] {};
    \node[] (n110) [above of=n19] {};
    \node[] (n111) [above of=n13] {};
    \node[] (n112) [right of=n14] {};
    \node[] (n113) [right of=n112] {};
    \node[] (n114) [above of=n111] {};
    
    \foreach \v/\vnext in {2/5, 3/6, 4/7, 8/10}{
      \draw [#1, -{Stealth[length=3mm, width=2mm]}] (n1\v) to[bend right=\bendangle] (n1\vnext); 
      \draw [#1, -{Stealth[length=3mm, width=2mm]}] (n1\vnext) to[bend left=\bendangle] (n1\v); 
    }
    \foreach \v/\vnext in {5/7, 1/8, 2/9, 3/10}{
      \draw [#2, -{Stealth[length=3mm, width=2mm]}] (n1\v) to[bend right=\bendangle] (n1\vnext); 
      \draw [#2, -{Stealth[length=3mm, width=2mm]}] (n1\vnext) to[bend left=\bendangle] (n1\v); 
    }
    \node[rectangle, rounded corners=0.8cm, draw=black!80, line width=0.5mm, inner sep=0pt, fit={(n12) (n13) (n14) (n110) (n111) (n112) (n113) (n114)}] (n1) {};

  }}
}

\tikzset{pics/.cd,
  pic InnerGraphMini/.style 2 args={code={
    \def\innerradius{0.8}
    \def\radius{3}
    \def\bendangle{5}

    \foreach \v/\lab in {1, 2, 3, 4, 5, 6, 7, 8}{  
      \node[inner sep=0pt,minimum size=0pt] (n1\v) at (60-\v*45: \innerradius cm) {};
    } 
    \foreach \v/\vnext in {1/4, 8/5}{
      \draw [#1, -{Stealth[length=1.5mm, width=2mm]}] (n1\v) to[bend right=\bendangle] (n1\vnext); 
      \draw [#1, -{Stealth[length=1.5mm, width=2mm]}] (n1\vnext) to[bend left=\bendangle] (n1\v); 
    }
    \foreach \v/\vnext in {7/2, 6/3}{
      \draw [#2, -{Stealth[length=1.5mm, width=2mm]}] (n1\v) to[bend right=\bendangle] (n1\vnext); 
      \draw [#2, -{Stealth[length=1.5mm, width=2mm]}] (n1\vnext) to[bend left=\bendangle] (n1\v); 
    }
    \node[rectangle, rounded corners=0.7cm, draw=black!80, line width=0.5mm, inner sep=0pt, fit={(n11) (n12) (n13) (n14) (n15) (n16) (n17) (n18)}] (n1) {};

  }}
}

\tikzset{pics/.cd,
  pic EmptyBox/.style 2 args={code={
   \def\innerradius{1.5}
    \def\radius{3}
    \def\innernodesize{0.15}
    
    \node[circle, minimum size=\innernodesize cm, inner sep=0pt, outer sep=0pt] (ne1) {};
    \node[circle, minimum size=\innernodesize cm, inner sep=0pt, outer sep=0pt]  (ne2) [right =#1 cm of ne1] {};
    \node[circle, minimum size=\innernodesize cm, inner sep=0pt, outer sep=0pt]  (ne3) [below =#2 cm of ne1] {};
    
    \node[rectangle, rounded corners=0.5cm, draw=black!80, line width=0.5mm, inner sep=10pt, fit={(ne1) (ne2) (ne3)}] (n1) {};
  }}
}

\newtheorem{theorem}{Theorem}

\usepackage{mathtools}

\begin{document}

\title{Directed graphs without rainbow stars\thanks{The work of the first author was supported by the National Research, Development and Innovation Office - NKFIH under the grants FK 132060 and KKP-133819. The work of the second and the fourth author was supported by the National Science Centre grant 2021/42/E/ST1/00193. The work of the third author was supported by a grant from the Simons Foundation \#712036.}}

\author{
Daniel Gerbner\thanks{Alfréd Rényi Institute of Mathematics, HUN-REN.  E-mail: {\tt gerbner.daniel@renyi.hu}.} \and
Andrzej Grzesik\thanks{Faculty of Mathematics and Computer Science, Jagiellonian University, {\L}ojasiewicza 6, 30-348 Krak\'{o}w, Poland. E-mail: {\tt Andrzej.Grzesik@uj.edu.pl}.}\and
Cory Palmer\thanks{Department of Mathematical Sciences, University of Montana. E-mail: \texttt{cory.palmer@umontana.edu}.} \and
Magdalena Prorok\thanks{AGH University of Krakow, al.~Mickiewicza 30, 30-059 Krakow, Poland. E-mail: {\tt prorok@agh.edu.pl}.}}

\date{}

\maketitle

\begin{abstract}
In a rainbow version of the classical Tur\'an problem one considers multiple graphs on a common vertex set, thinking of each graph as edges in a distinct color, and wants to determine the minimum number of edges in each color which guarantees existence of a rainbow copy (having at most one edge from each graph) of a given graph. 
Here, we prove an optimal solution for this problem for any directed star and any number of colors.
\end{abstract}

\section{Introduction}

One of the central topics in extremal graph theory, known as the Tur\'an problem, is to determine the maximum number of edges of a graph on $n$ vertices that does not contain a copy of a given graph $F$ as a subgraph. Equivalently, the minimum number of edges that forces the existence of $F$ as a subgraph. Research on this topic and its various generalizations, provides a deep understanding of the relationship between various global properties and local structures of graphs.

Recently, a rainbow version of this problem has been intensively studied. In this variant, for an integer $c \geq 1$ we consider a collection of $c$ graphs $\mathcal{G} = (G_1, \ldots, G_c)$ on a common vertex set and say that a graph $F$ is a \emph{rainbow subgraph} of $\mathcal{G}$ (or $\mathcal{G}$ \emph{contains} $F$) if there exists an injective function $\varphi : E(F) \to [c]$ such that for each $e \in E(F)$ it holds $e \in G_{\varphi(e)}$. In other words, we think of each graph in $\mathcal{G}$ as edges in a distinct color and $\mathcal{G}$ as a $c$-edge-colored multigraph with each color spanning a simple graph.
We want to force the existence of a rainbow copy of $F$ in~$\mathcal{G}$ by having a large number of edges in each graph. 
Typically, by bounding the value of $\min_{1\leq i\leq c} e(G_i)$ \cite{ADGMS20, BG22}, i.e., the number of edges in each graph, or the sum $\sum_{1\leq i\leq c} e(G_i)$ \cite{CKLLS22, KSSV04}, but other measures are also considered, in particular the product $\prod_{1\leq i\leq c} e(G_i)$ \cite{Fra22, FGHLSTVZ22}, or more general functions of the number of edges \cite{FMR22, IKLS23}.

Such a rainbow Tur\'an problem was also considered for directed graphs in \cite{BGP23}, where the optimal solution (up to a lower order error term) for $\min_{1\leq i\leq c} e(G_i)$ and $\sum_{1\leq i\leq c} e(G_i)$ was provided, for any number of colors, when a directed or transitive rainbow triangle is forbidden.  
Here, we continue this line of research on directed graphs and consider a directed star as the forbidden rainbow graph. 

Let $S_{p,q}$ be the orientation of a star on $p+q+1$ vertices with \emph{center} vertex of indegree~$p$ and outdegree $q$.
Forbidding a rainbow $S_{p,q}$ in a collection of graphs $\mathcal{G} = (G_1, \ldots, G_c)$ is analogous to forbidding a rainbow $S_{q,p}$ in the collection of graphs obtained by changing the orientation of every edge in each graph from $\mathcal{G}$. Thus it is enough to consider this rainbow Tur\'an problem for $S_{p,q}$ only when $p \leq q$. As this problem is trivial when the number of colors $c$ is less than the number of edges in the forbidden rainbow graph, we consider only $c \geq p+q$.  

In Section~\ref{sec:S_0,q} we consider a star $S_{0,q}$ as the forbidden rainbow graph and prove, for every $n > c \geq q\geq 1$, exact bounds for $\sum_{i=1}^c e(G_i)$ (Theorem~\ref{thm:S_0,q_sum}) and $\min_{1\leq i\leq c} e(G_i)$ (Theorem~\ref{thm:S_0,q}).
In Section~\ref{sec:S_p,q} we consider $S_{p,q}$ as the forbidden rainbow graph for any $q \geq p \geq 1$ and prove bounds for $\sum_{i=1}^c e(G_i)$ (Theorem~\ref{thm:S_p,q_sum}) and $\min_{1\leq i\leq c} e(G_i)$ (Theorem~\ref{thm:S_p,q}), which are tight up to a lower order error term.
Additionally, in Section~\ref{sec:S_1,1} we provide exact bounds for any $c\geq2$ and $n \geq 3$ when a rainbow $S_{1,1}$, i.e., directed path of length 2, is forbidden, for both the sum (Theorem~\ref{thm:S_1,1_sum}) and the minimum (Theorem~\ref{thm:S_1,1_2+}) of the number of edges. 

\section{Rainbow directed $S_{0,q}$}\label{sec:S_0,q}

In this section the forbidden rainbow graph is a directed star with all edges oriented away from the center. As noted earlier, the problem is the same if we forbid a directed star with all edges oriented to the center. 
The following theorem provides the optimal bound for the sum of the number of edges in all graphs.

\begin{theorem}\label{thm:S_0,q_sum}
For  integers $n > c \geq q \geq 1$, every collection of directed graphs $G_1, \ldots, G_c$ on a common set of $n$ vertices containing no rainbow $S_{0,q}$ satisfies \[\sum_{i=1}^c e(G_i) \leq (q-1)(n^2-n).\]
Moreover, this bound is sharp.
\end{theorem}

\begin{figure}[ht]
\centering
\begin{tikzpicture}[scale=0.6]
\def\radius{4}
\def\baseangle{30}
\def\bendangle{10}

\pic [local bounding box=A1] at (\baseangle-1*120:\radius cm) {pic InnerGraph={green}{red}};
\pic [local bounding box=A2] at (\baseangle-2*120:\radius cm) {pic InnerGraph={red}{blue}};
\pic [local bounding box=A3] at (\baseangle-3*120:\radius cm) {pic InnerGraph={blue}{green}};

\draw [green, -{Stealth[length=3mm, width=2mm]}, shorten >=1.05cm, shorten <=1.05cm] (A1.center) to[bend left=\bendangle] (A2.center); 
\draw [blue, -{Stealth[length=3mm, width=2mm]}, shorten >=1.05cm, shorten <=1.05cm] (A2.center) to[bend left=\bendangle] (A1.center); 
\draw [red, -{Stealth[length=3mm, width=2mm]}] (A1) to (A2); 
\draw [red, -{Stealth[length=3mm, width=2mm]}] (A2) to (A1); 

\draw [red, -{Stealth[length=3mm, width=2mm]}] (A2) to[bend left=\bendangle] (A3); 
\draw [green, -{Stealth[length=3mm, width=2mm]}] (A3) to[bend left=\bendangle] (A2); 
\draw [blue, -{Stealth[length=3mm, width=2mm]}] (A2) to (A3); 
\draw [blue, -{Stealth[length=3mm, width=2mm]}] (A3) to (A2); 

\draw [red, -{Stealth[length=3mm, width=2mm]}, shorten >=1.05cm, shorten <=1.05cm] (A1.center) to[bend left=\bendangle] (A3.center); 
\draw [blue, -{Stealth[length=3mm, width=2mm]}, shorten >=1.05cm, shorten <=1.05cm] (A3.center) to[bend left=\bendangle] (A1.center); 
\draw [green, -{Stealth[length=3mm, width=2mm]}] (A1) to (A3); 
\draw [green, -{Stealth[length=3mm, width=2mm]}] (A3) to (A1); 

\end{tikzpicture}
\caption{The optimal construction for 3 colors and forbidden rainbow $S_{0,3}$.}\label{fig:S_0,3}
\end{figure}
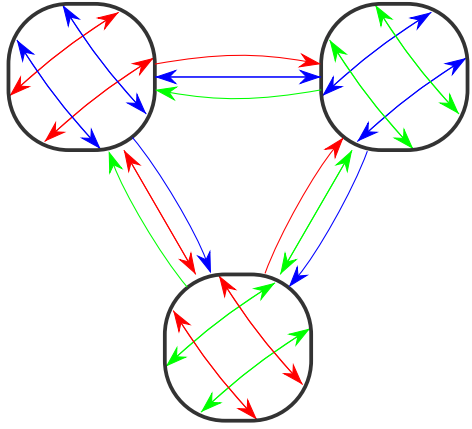

\begin{proof}
The sharpness of the bound follows from taking a collection of $q-1$ complete graphs, but there are more extremal constructions. 
We split the vertex set into disjoint subsets and \emph{assign} to each of them a different set of $q-1$ colors. 
This way each vertex has $q-1$ colors assigned. 
Then, for any two vertices $u$, $v$ and color $i \in [c]$, we add edge $uv$ to $G_i$ if color $i$ is assigned to $u$. 
Clearly this construction does not contain a rainbow $S_{0,q}$ as no vertex has positive outdegree in $q$ graphs, and each vertex has the sum of outdegrees over all graphs equal to $(q-1)(n-1)$ giving the total number of edges equal to $(q-1)(n^2-n)$.
An example construction of this type is shown in Figure~\ref{fig:S_0,3}.

To prove the upper bound, consider a collection of graphs $G_1, G_2, \ldots, G_c$ on a common set $V$ of $n$ vertices that does not contain a rainbow $S_{0,q}$. 
Let $v \in V$ be an arbitrary vertex. We will show that the total number of edges outgoing from $v$ is bounded by $(q-1)(n-1)$. 
Let~$H$~be an auxiliary bipartite graph between the vertices $V\setminus\{v\}$ and colors in $[c]$ created by connecting vertex $u \in V\setminus\{v\}$ and color $i \in [c]$ if $vu \in E(G_i)$.
The number of edges in~$H$ is equal to the number of outgoing edges from $v$ that we want to bound.
Note that the existence of a matching of size $q$ in $H$ means that there exists a rainbow $S_{0,q}$ with center $v$, which is not possible.
Therefore, the maximum matching in $H$ is of size at most $q-1$. 
From K\"onig's theorem this means that the minimum vertex cover is also of size at most $q-1$. 
As $n-1 \geq c$, this implies that the maximum number of edges in $H$ is bounded by $(q-1)(n-1)$, as desired. 
By summing it over all vertices we obtain $\sum_{i=1}^c e(G_i) \leq (q-1)(n^2-n)$.
\end{proof}

Note that the bound $n>c$ in Theorem~\ref{thm:S_0,q} is indeed needed, as otherwise the following collections of graphs contradict the theorem. 
If $c \geq n \geq q$, for each vertex $v$, add edges in each color from $v$ to $q-1$ other arbitrary vertices. This way \[\sum_{i=1}^c e(G_i) = (q-1)cn > (q-1)(n^2-n)\] and there is no rainbow $S_{0,q}$. 
While if $n\le q$, then a collection of complete directed graphs does not contain rainbow $S_{0,q}$ and satisfies \[\sum_{i=1}^c e(G_i) = c(n^2-n) > (q-1)(n^2-n).\]
The same constructions also provide an exact bound for $\min_{1\leq i \leq c} e(G_i)$ for any $n \leq c$.

Theorem~\ref{thm:S_0,q_sum} implies that for integers $n > c \geq q \geq 1$, every collection of directed graphs $G_1, \ldots, G_c$ on a common set of $n$ vertices containing no rainbow $S_{0,q}$ satisfies \[\min_{1\leq i \leq c} e(G_i) \leq \frac{q-1}{c}\left(n^2-n\right).\] 
Moreover, if $n(q-1)$ is divisible by $c$, then one can make a construction as detailed in the beginning of the proof of Theorem~\ref{thm:S_0,q_sum}, in which the number of edges in each graph is the same (see Figure~\ref{fig:S_0,3}). 
This means that for such $n$ the above bound is sharp. 
We can actually extend this to obtain the optimal bound for every $n>c$.

\begin{theorem}\label{thm:S_0,q}
For integers $n > c \geq q \geq 1$, every collection of directed graphs $G_1, \ldots, G_c$ on a common set of $n$ vertices containing no rainbow $S_{0,q}$ satisfies 
\[\min_{1\leq i \leq c} e(G_i) \leq \left\lfloor\frac{n(q-1)}{c}\right\rfloor(n-1) + r,\]
where $r$ is the remainder of $n(q-1)$ when divided by $c$.
Moreover, this bound is sharp.
\end{theorem}

\begin{proof}
We proceed similarly as in the proof of Theorem~\ref{thm:S_0,q_sum}. 
Consider a collection of graphs $G_1, G_2, \ldots, G_c$ on a common set $V$ of $n$ vertices that does not contain a rainbow $S_{0,q}$. 
For any vertex $v \in V$ we consider an auxiliary bipartite graph $H$ between the vertices $V\setminus\{v\}$ and colors in $[c]$ created by connecting vertex $u \in V\setminus\{v\}$ and color $i \in [c]$ if $vu \in E(G_i)$.
Since the existence of a matching of size $q$ in $H$ means that there exists a forbidden rainbow $S_{0,q}$ centered in $v$, the maximum matching in $H$ is of size $q-1$.
From K\"onig's theorem the minimum vertex cover is also of size at most $q-1$.
Let $a_v$ and $b_v$ be the numbers of vertices in the minimum vertex cover that are in the part of $H$ related with the colors and in the part related with the other vertices, respectively.
In particular, in $a_v$ colors there are at most $n-1$ edges outgoing from~$v$ and in all other colors there are at most $b_v$ edges outgoing from~$v$.  
Let $a = \sum_{v\in V} a_v$ and $b = \sum_{v\in V} b_v$.
Since $a_v + b_v \leq q-1$, we have $a + b \leq n(q-1)$.

There exists a color $i$ which appears at most $\left\lfloor\frac{a}{c}\right\rfloor$ times in the minimum vertex covers, which gives that
\[e(G_i) \leq \left\lfloor\frac{a}{c}\right\rfloor(n-1) + b \leq \left\lfloor\frac{n(q-1)-b}{c}\right\rfloor(n-1) + b.\]
Note that increasing $b$ by $1$ either increases the above bound by $1$ or decreases it by $n-2$, and the decrease happens after at most $c-1$ increases. 
Since $n>c$, the maximum value of the bound is achieved in the last moment before the first decrease, which occurs for $b$ equal to the remainder of $n(q-1)$ when divided by $c$, as desired. 

The sharpness of the bound follows from a modification of a construction described in the proof of Theorem~\ref{thm:S_0,q_sum}. 
We enumerate the vertex set by consecutive integers from $1$ to $n$ and assign to every vertex $j \in [n]$ all colors from $(j-1)(q-1)$ to $j(q-1)-1$ considered modulo $c$. 
This way each vertex has $q-1$ colors assigned and in total we have $n(q-1)$ assignments. 
Now, remove the last $r$ assignments. 
Since $r$ is the remainder of $n(q-1)$ when divided by $c$, after the removal each color was assigned the same number of times.
For every $i \in [c]$, add to $G_i$ an edge from each vertex with color $i$ assigned to any other vertex.
This gives $\left\lfloor\frac{n(q-1)}{c}\right\rfloor(n-1)$ edges in each graph. 
Note that every vertex $v$ having $x$ assignments removed has positive outdegree in exactly $q-1-x$ graphs, so we may still add edges from $v$ to arbitrary $x$ vertices and avoid creating a rainbow $S_{0,q}$. 
This addition increases the number of edges in each graph by the total number of removed assignments, which is equal to $r$.
Altogether we obtain the required number of edges in each graph.
\end{proof}

\section{General rainbow directed star}\label{sec:S_p,q}

In this section, for integers $p,q,c\geq1$, we consider a collection of $c$ directed graphs $G_1, \ldots, G_c$ on a common vertex set and forbid rainbow star $S_{p,q}$, which is a star with the center of indegree~$p$ and outdegree $q$.
Since the problem is trivial if $c <p+q$ and symmetric with respect to $p$ and $q$, it is enough to consider only $c\geq p+q$ and $p \leq q$. 
The following theorem provides the optimal bound for the total number of edges in all graphs.

\begin{theorem}\label{thm:S_p,q_sum}
For integers $q \geq p \geq 1$, $c \geq p+q$ and $n$, every collection of directed graphs $G_1, \ldots, G_c$ on a common set of $n$ vertices containing no rainbow $S_{p,q}$ satisfies
\[\sum_{i=1}^c e(G_i) \leq 
\begin{cases} 
(p+q-1)n^2 + o(n^2) & \ \mathrm{if}\quad c \leq p+2q-1+2\sqrt{pq},\\[2pt]
\left(\frac{(c-p+1)^2}{4(c-q+1)} + p-1\right)n^2 + o(n^2) & \ \mathrm{if}\quad c \geq p+2q-1+2\sqrt{pq}.
\end{cases}
\]
Moreover, the above bounds are tight up to a lower order error term.
\end{theorem}

\begin{proof}
The claimed bound is tight in the first case when the collection consists of exactly $p+q-1$ complete directed graphs and all other graphs are empty. 
This construction obviously has no rainbow $S_{p,q}$. 
In the second case, as the first $p-1$ graphs we take complete directed graphs, while for the remaining graphs we split the vertex set into two disjoint sets $A$ of size $\frac{c-p+1}{2(c-q+1)}n$ and $C$ of size $\frac{c+p-2q+1}{2(c-q+1)}n$.
For $i \in [q-1]\setminus[p-1]$, in graph $G_i$ we add all edges inside $A$ and from $C$ to $A$. 
While for $i \in [c]\setminus[q-1]$, in graph $G_i$ we add all edges from $C$ to $A$. 
Note that vertices in~$A$ have nonzero outdegree only in $q-1$ graphs, while vertices in $C$ have nonzero indegree in $p-1$ graphs, so none of them can be the center of a rainbow $S_{p,q}$. 
The total number of edges in all graphs in such construction is equal to 
\[(p-1)n^2 + |A|^2(q-p) + |A||C|(c-p+1) + o(n^2) = \left(\frac{(c-p+1)^2}{4(c-q+1)} + p-1\right)n^2 + o(n^2),\]
as desired.

To prove the upper bound, consider a collection of directed graphs $G_1, \ldots, G_c$ on a common set $V$ of $n$ vertices containing no rainbow $S_{p,q}$. 
Note that a vertex in $V$ can have nonzero indegree in $p$ graphs and nonzero outdegree in different $q$~graphs, for example if a pair of vertices is connected by edges in all the graphs in both directions. 
Or more generally, if a vertex is the center of a noninjective homomorphic image of a rainbow $S_{p,q}$.
Using a colored graph removal lemma implied by the Szemer\'edi Regularity Lemma (see \cite{Fox11} or Appendix C in~\cite{IKLS23}), we may remove all homomorphic images of a rainbow $S_{p,q}$ by deleting $o(n^2)$ edges in total. 
Thus, we may assume that no vertex in $V$ has nonzero indegree in $p$ graphs and nonzero outdegree in $q$~different graphs.

We split the vertex set $V$ into three disjoint sets. 
Let $B$ be the set of vertices incident to edges in at most $p+q-1$ graphs, $A$ be the set of vertices in $V\setminus B$ that have nonzero outdegree in at most $q-1$ graphs, and $C$ be the set of vertices in $V\setminus B$ having nonzero indegree in at most $p-1$ graphs. 
Let us denote $\alpha=|A|$, $\beta=|B|$ and $\gamma=|C|$.

Note that any two vertices in $A$ may be connected by at most $2(q-1)$ edges ($q-1$ in each direction), vertices in $B$ by at most $2(p+q-1)$ edges, while vertices in $C$ by at most $2(p-1)$ edges.
Additionally, between any vertices in $A$ and $B$ we have at most $p+2q-2$ edges ($q-1$ from $A$ to $B$ and $p+q-1$ from $B$ to $A$), between vertices in $B$ and $C$ there are at most $2p+q-2$ edges, while between vertices in $A$ and $C$ we have at most $c+p-1$ edges.
This gives an upper bound for the total number of edges in all the graphs of\\
\vspace{0pt} 
\begin{align*}
& (q-1)\alpha^2 + (p+q-1)\beta^2 + (p-1)\gamma^2 + (p+2q-2)\alpha\beta + (2p+q-2)\beta\gamma + (c+p-1)\alpha\gamma\\
&= \left(\frac{(c-p+1)^2}{4(c-q+1)} + p-1\right)(\alpha+\gamma)^2 - (c-q+1)\left(\frac{c-p+1}{2(c-q+1)}(\alpha+\gamma)-\alpha\right)^2 \\
&\qquad + (p+q-1)\beta^2 + (p+2q-2)\alpha\beta + (2p+q-2)\beta\gamma\\
&\leq \left(\frac{(c-p+1)^2}{4(c-q+1)} + p-1\right)(\alpha+\gamma)^2 + (p+q-1)\beta^2 + (p+2q-2)\beta(\alpha+\gamma)\\
&= \left(\frac{(c-p+1)^2}{4(c-q+1)} + p-1\right)(n-\beta)^2 + (p+q-1)\beta^2 + (p+2q-2)\beta(n-\beta).
\end{align*}

The obtained bound is a quadratic function of $\beta \in [0,n]$ with the coefficient of $\beta^2$ equal to $\frac{(c+p-2q+1)^2}{4(c-q+1)}$, so it is convex and its maximum is reached for $\beta=0$ or $\beta=n$. It is easy to verify that for $c \leq p+2q-1+2\sqrt{pq}$ the maximum occurs when $\beta=n$, while for $c \geq p+2q-1+2\sqrt{pq}$ it occurs when $\beta=0$, which gives the desired bounds.
\end{proof}

Note that the bound in Theorem~\ref{thm:S_p,q_sum} is not realizable for $p\geq2$ in any collection of directed graphs each having the same number of edges. 
Thus, for $p\geq2$ the optimal bound for $\min_{1\leq i\leq c} e(G_i)$ is different. 
On the other hand, when $p=1$ the second bound in Theorem~\ref{thm:S_p,q_sum} can be obtained in such a collection of directed graphs, so this theorem implies the optimal bound for $\min_{1\leq i\leq c} e(G_i)$ for $c \geq 2q+2\sqrt{q}$. 
It occurs that the same construction gives the optimal bound already for $c \geq \max\{q+\sqrt{q},2q-2\}$, but for smaller values of $c$ there are different optimal constructions. 

The theorem below gives the optimal bound for $\min_{1\leq i\leq c} e(G_i)$ for any integers $p$ and $q$. 

\begin{theorem}\label{thm:S_p,q}
For integers $q \geq p \geq 1$, $c \geq p+q$ and $n$, every collection of directed graphs $G_1, \ldots, G_c$ on a common set of $n$ vertices containing no rainbow $S_{p,q}$ satisfies the following.
The values
\[t_1 = 2p+q-1, \qquad 
t_2 = \begin{cases} \frac{(q-1)(p+q-1)}{q-p-1} & \ \mathrm{if}\quad q \geq p+2,\\ \infty & \ \mathrm{if}\quad q \leq p+1,\end{cases}\]
\[t_3 = p+q-1+\sqrt{pq}, \qquad 
t_4 = q-1+\sqrt{(q-1)(q-p)}\]
satisfy $t_1 \leq t_2$, $t_1 \leq t_3$, and either $t_2 \leq t_3 \leq t_4$ or $t_4 \leq t_3 \leq t_2$. \\
If $t_2 \leq t_3 \leq t_4$, then
\[\min_{1 \leq i \leq c} e(G_i) \leq 
\begin{cases} 
\frac{(p+q-1)^2}{c^2}n^2 + o(n^2) & \ \mathrm{if}\quad c \leq t_1,\\[3pt]
\frac{(c-q+1)^2(p+q-1)^2}{4c^2p(c-p-q+1)}n^2+o(n^2) &\ \mathrm{if}\quad  t_1 \leq c \leq t_2,\\[3pt]
\frac{q-1}{c}n^2 + o(n^2) & \ \mathrm{if}\quad t_2 \leq c \leq t_4,\\[3pt]
\frac{(c^2-(p-1)(q-1))^2}{4c^2(c-p+1)(c-q+1)}n^2 + o(n^2) & \ \mathrm{if}\quad c \geq t_4.
\end{cases}\]
While if $t_4 \leq t_3 \leq t_2$, then
\[\min_{1 \leq i \leq c} e(G_i) \leq 
\begin{cases} 
\frac{(p+q-1)^2}{c^2}n^2 + o(n^2) & \ \mathrm{if}\quad c \leq t_1,\\[3pt]
\frac{(c-q+1)^2(p+q-1)^2}{4c^2p(c-p-q+1)}n^2+o(n^2) &\ \mathrm{if}\quad  t_1 \leq c \leq t_3,\\[3pt]
\frac{(c^2-(p-1)(q-1))^2}{4c^2(c-p+1)(c-q+1)}n^2 + o(n^2) & \ \mathrm{if}\quad c \geq t_3.
\end{cases}\]
Moreover, the above bounds are tight up to a lower order error term.
\end{theorem}

\begin{proof}
First, we prove the inequalities between thresholds $t_i$ for $i\in[4]$. 
Note that $t_3 \geq t_1$ since $q\geq p$.
If $q\leq p+1$, then clearly $t_4 \leq t_3 \leq t_2$ and $t_1 \leq t_2$. 
While for $q \geq p+2$ we have
\[t_2 = \frac{(q-1)(p+q-1)}{q-p-1} = \frac{(q-p-1+p)(q-p-1+2p)}{q-p-1} = q+2p-1+\frac{2p^2}{q-p-1} \geq t_1.\]
One can also verify that for $q \geq p+2$ each of the inequalities $t_2\leq t_3$ and $t_3 \leq t_4$ is equivalent to the inequality $q(q-p-1)^2 \geq p(p+q-1)^2$, so either $t_2 \leq t_3 \leq t_4$ or $t_4 \leq t_3 \leq t_2$, as desired.

Consider a collection of directed graphs $G_1, \ldots, G_c$ on a common set $V$ of $n$ vertices containing no rainbow $S_{p,q}$. 
Similarly as in the proof of Theorem~\ref{thm:S_p,q_sum} we use a colored graph removal lemma to remove all homomorphic images of a rainbow $S_{p,q}$ by deleting $o(n^2)$ total edges. 
Thus, we may assume that no vertex in $V$ has nonzero indegree in $p$ graphs and nonzero outdegree in $q$ different graphs.

We split the vertex set $V$ into disjoint sets. 
Let $B$ be the set of vertices incident to edges in at most $p+q-1$ graphs, $A$ be the set of vertices in $V\setminus B$ that have nonzero outdegree in at most $q-1$ graphs, and $C$ be the set of vertices in $V\setminus B$ having nonzero indegree in at most $p-1$ graphs. 
Additionally, for each $i \in [c]$, let $A_i \subset A$ be the set of vertices in $A$ that have nonzero outdegree in $G_i$, similarly $C_i \subset C$ be the set of vertices in $C$ that have nonzero indegree in $G_i$, while $B_i \subset B$ be the set of vertices in $B$ incident to edges in $G_i$.
For every $i\in [c]$, we denote $\alpha_i=|A_i|$, $\beta_i=|B_i|$, $\gamma_i=|C_i|$, 
$\alpha=|A|$, $\beta=|B|$ and $\gamma=|C|$. 

Observe that for every $i \in [c]$, 
\begin{equation}\label{eq:S_p,q_edgenum}
e(G_i) \leq (\alpha_i+\beta_i+\gamma)(\alpha+\beta_i+\gamma_i),
\end{equation}
because edges of $G_i$ can go only from vertices in $A_i \cup B_i \cup C$ to vertices in $A \cup B_i \cup C_i$.

For integers $x, y \geq 1$, by averaging over all colors, there is $j\in [c]$ such that
\[y\alpha + x\alpha_j + (x+y)\beta_j + x\gamma + y\gamma_j \leq \frac{1}{c}\sum_{i=1}^c y\alpha + x\alpha_i + (x+y)\beta_i + x\gamma + y\gamma_i.\]
Since 
\[\sum_{i=1}^c \alpha_i \leq (q-1)\alpha, \qquad \sum_{i=1}^c \beta_i \leq (p+q-1)\beta \qquad \mathrm{and} \qquad \sum_{i=1}^c \gamma_i \leq (p-1)\gamma,\]
from $(\ref{eq:S_p,q_edgenum})$ we obtain
\begin{align}\label{eq:S_p,q_jbound}
e(G_j) &\leq \frac{1}{xy}(x\alpha_j+x\beta_j+x\gamma)(y\alpha+y\beta_j+y\gamma_j) \nonumber\\
&\leq \frac{1}{4xy}\big(y\alpha + x\alpha_j + (x+y)\beta_j + x\gamma + y\gamma_j\big)^2 \nonumber\\
&\leq \frac{1}{4c^2xy}\big((yc+x(q-1))\alpha + (x+y)(p+q-1)\beta + (xc + y(p-1))\gamma\big)^2.
\end{align}

In particular, for $x=1$ and $y=1$, since $p \leq q$ and $\alpha+\gamma=n-\beta$, this gives
\begin{align*}
e(G_j) &\leq \frac{1}{4c^2}\big((c+q-1)\alpha + 2(p+q-1)\beta + (c+p-1)\gamma\big)^2 \\
&\leq \frac{1}{4c^2}\big((c+q-1)(\alpha+\gamma) + 2(p+q-1)\beta \big)^2 \\
&= \frac{1}{4c^2}\big((c+q-1)n + (2p+q-1-c)\beta\big)^2.
\end{align*}
If $c \leq t_1$, then the above expression is maximized at $\beta=n$ and we obtain
\[e(G_j) \leq \frac{(p+q-1)^2}{c^2}n^2,\]
which gives the first bound in both cases of the theorem. 
This bound is achieved if $A=\emptyset$, $C=\emptyset$ and $B$ is divided into $\binom{c}{p+q-1}$ equal-sized sets, one for each subset of $p+q-1$ colors assigned to the set, with edges in $G_i$, for $i\in [c]$, between any vertices from sets having color $i$ assigned. 
This is depicted in Figure~\ref{fig:S_1,2_3} for $p=1$, $q=2$ and $c=3$.

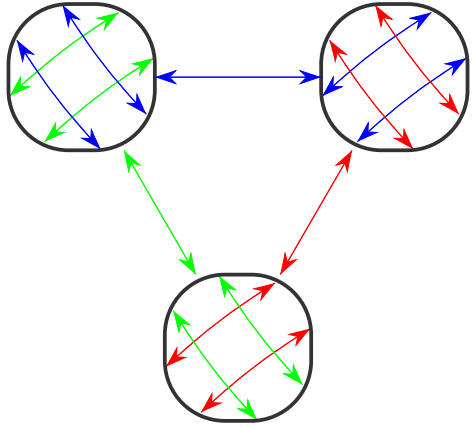
\begin{figure}[ht]
\centering
\begin{tikzpicture}[scale=0.6]
\def\radius{4}
\def\baseangle{30}

\def\bendangle{10}

\pic [local bounding box=A1] at (\baseangle-1*120:\radius cm) {pic InnerGraph={red}{green}};
\pic [local bounding box=A2] at (\baseangle-2*120:\radius cm) {pic InnerGraph={green}{blue}};
\pic [local bounding box=A3] at (\baseangle-3*120:\radius cm) {pic InnerGraph={blue}{red}};

\draw [green, -{Stealth[length=3mm, width=2mm]}] (A1) to (A2);
\draw [green, -{Stealth[length=3mm, width=2mm]}] (A2) to (A1);
\draw [red, -{Stealth[length=3mm, width=2mm]}] (A1) to (A3);
\draw [red, -{Stealth[length=3mm, width=2mm]}] (A3) to (A1);
\draw [blue, -{Stealth[length=3mm, width=2mm]}] (A2) to (A3);
\draw [blue, -{Stealth[length=3mm, width=2mm]}] (A3) to (A2);

\end{tikzpicture}
\caption{The optimal construction for 3 colors and a forbidden rainbow $S_{1,2}$.}\label{fig:S_1,2_3}
\end{figure}

Consider now $c \geq t_1$ and take $x=c-p-q+1$ and $y=p$. 
This gives 
\[yc+x(q-1) = (x+y)(p+q-1) = (c-q+1)(p+q-1),\] 
so from $(\ref{eq:S_p,q_jbound})$ we obtain
\begin{align*}
e(G_j) &\leq \frac{1}{4c^2p(c\!-\!p\!-\!q\!+\!1)}\big((c-q+1)(p+q-1)(\alpha+\beta) + ((c-p-q+1)c + p(p-1))\gamma\big)^2\\
&= \frac{1}{4c^2p(c\!-\!p\!-\!q\!+\!1)}\big((c-q+1)(p+q-1)n - (pq - (c-p-q+1)^2)\gamma\big)^2.
\end{align*}
If $c \leq t_3$, then the above expression is maximized at $\gamma=0$ and gives
\[e(G_j) \leq \frac{(c-q+1)^2(p+q-1)^2}{4c^2p(c-p-q+1)}n^2.\]
This bound is achieved when $C = \emptyset$, the set $B$ is of size $\frac{(q-1)(p+q-1)-c(q-p-1)}{2p(c-p-q+1)}n$ and is divided into $\binom{c}{p+q-1}$ equal-sized sets, one for each subset of $p+q-1$ colors assigned to the set, while set $A$ is of size $\frac{(p+q-1)(c-2p-q+1)}{2p(c-p-q+1)}n$ and is divided into $\binom{c}{q-1}$ equal-sized sets, one for each subset of $q-1$ colors assigned to the set. 
We put edges in $G_i$, for $i\in [c]$, between any vertex from a set in $A\cup B$ having color $i$ assigned to any other vertex in $A$ or a vertex in $B$ having color $i$ assigned. 
This construction is possible only if the mentioned sizes of sets $A$ and $B$ are non-negative, which occurs when $c \geq t_1$ and $c \leq t_2$, and so it proves the second bound in both cases of the theorem.
Note that if $q \leq p+1$, then the size of $B$ is always positive, which explains why $t_2$ is defined in this way.
This construction is illustrated in Figure~\ref{fig:S_1,2_3AB} for $p=1$, $q=2$ and $c=3$, despite the fact that in this case the optimal size of $B$ is equal to $0$.

\begin{figure}[ht]
\centering
\begin{tikzpicture}[scale=0.5]
\def\radius{10}
\def\basedist{7}
\def\baseangle{40}
\def\shiftlvl{0.4}
\def\bendangle{10}
\def\bendangletwo{5}


\node[rectangle, draw=red!0, line width=0.5mm, minimum width=1.65cm, minimum height = 1.65cm, rounded corners=0.cm] (A1) at (0,0) {};
\pic [local bounding box=A11] at (0,0) {pic InnerGraphMini={red}{red}};

\node[rectangle, draw=red!0, line width=0.5mm, minimum width=1.65cm, minimum height = 1.65cm, rounded corners=0.cm] (A2) at (\basedist, 0) {};
\pic [local bounding box=A21] at (\basedist, 0) {pic InnerGraphMini={blue}{blue}};

\node[rectangle, draw=red!0, line width=0.5mm, minimum width=1.65cm, minimum height = 1.65cm, rounded corners=0.cm] (A3) at (2*\basedist, 0) {};
\pic [local bounding box=A31] at (2*\basedist, 0) {pic InnerGraphMini={green}{green}};


\node[rectangle, draw=red!0, line width=0.5mm, minimum width=1.95cm, minimum height = 1.95cm, rounded corners=0.cm] (A4) at (0, -\basedist) {};
\pic [local bounding box=A41] at (0, -\basedist) {pic InnerGraph={red}{blue}};

\node[rectangle, draw=red!0, line width=0.5mm, minimum width=1.95cm, minimum height = 1.95cm, rounded corners=0.cm] (A5) at (\basedist, -\basedist) {};
\pic [local bounding box=A51] at (\basedist, -\basedist) {pic InnerGraph={green}{red}};

\node[rectangle, draw=red!0, line width=0.5mm, minimum width=1.95cm, minimum height = 1.95cm, rounded corners=0.cm] (A6) at (2*\basedist, -\basedist) {};
\pic [local bounding box=A61] at  (2*\basedist, -\basedist) {pic InnerGraph={blue}{green}};

\draw [red, -{Stealth[length=3mm, width=2mm]}] ([shift=({-0.4 cm, -1.6 cm})]A1.center) to[bend left=-1.5*\bendangle] ([shift=({-0.4 cm, 2 cm})]A4.center);
\draw [red, -{Stealth[length=3mm, width=2mm]}] ([shift=({-0.4 cm, 2 cm})]A4.center) to[bend right=-1.5*\bendangle] ([shift=({-0.4 cm, -1.6 cm})]A1.center);
\draw [blue, -{Stealth[length=3mm, width=2mm]}] ([shift=({-0.1 cm, 2 cm})]A4.center) to[bend left=-1.5*\bendangle] ([shift=({-0.1 cm, -1.6 cm})]A1.center);

\draw [red, -{Stealth[length=3mm, width=2mm]}] ([shift=({-0.2 cm, 2 cm})]A5.center) to[bend right=-1.5*\bendangle] ([shift=({-0.2 cm, -1.6 cm})]A2.center);
\draw [green, -{Stealth[length=3mm, width=2mm]}] ([shift=({0.1 cm, 2 cm})]A5.center) to[bend left=-1.5*\bendangle] ([shift=({0.1 cm, -1.6 cm})]A2.center);

\draw [blue, -{Stealth[length=3mm, width=2mm]}] ([shift=({0 cm, 2 cm})]A6.center) to[bend left=1.5*\bendangle] ([shift=({0 cm, -1.6 cm})]A3.center);
\draw [green, -{Stealth[length=3mm, width=2mm]}] ([shift=({0.3 cm, 2 cm})]A6.center) to[bend right=1.5*\bendangle] ([shift=({0.3 cm, -1.6 cm})]A3.center);
\draw [green, -{Stealth[length=3mm, width=2mm]}] ([shift=({0.3 cm, -1.6 cm})]A3.center) to[bend left=1.5*\bendangle] ([shift=({0.3 cm, 2 cm})]A6.center);

\draw [red, -{Stealth[length=3mm, width=2mm]}] ([shift=({1.6 cm, 0.4 cm})]A1.center) to[bend left=1.5*\bendangle] ([shift=({-1.6 cm, 0.4 cm})]A2.center);
\draw [blue, -{Stealth[length=3mm, width=2mm]}] ([shift=({-1.6 cm, 0.1 cm})]A2.center) to[bend left =1.5*\bendangle] ([shift=({1.6 cm, 0.1 cm})]A1.center);

\draw [blue, -{Stealth[length=3mm, width=2mm]}] ([shift=({1.6 cm, 0.4 cm})]A2.center) to[bend left=1.5*\bendangle] ([shift=({-1.6 cm, 0.4 cm})]A3.center);
\draw [green, -{Stealth[length=3mm, width=2mm]}] ([shift=({-1.6 cm, 0.1 cm})]A3.center) to[bend left =1.5*\bendangle] ([shift=({1.6 cm, 0.1 cm})]A2.center);

\draw [red, -{Stealth[length=3mm, width=2mm]}] ([shift=({1. cm, 1.5 cm})]A1.center) to[bend left=3*\bendangle] ([shift=({-1. cm, 1.5 cm})]A3.center);
\draw [green, -{Stealth[length=3mm, width=2mm]}] ([shift=({-1.2 cm, 1.2 cm})]A3.center) to[bend left=-3*\bendangle] ([shift=({1.2 cm, 1.2 cm})]A1.center);

\draw [red, -{Stealth[length=3mm, width=2mm]}] ([shift=({2 cm, 0.1 cm})]A4.center) to[bend left=-1.5*\bendangle] ([shift=({-2 cm, 0.1 cm})]A5.center);
\draw [red, -{Stealth[length=3mm, width=2mm]}] ([shift=({-2 cm, 0.1 cm})]A5.center) to[bend right =-1.5*\bendangle] ([shift=({2 cm, 0.1 cm})]A4.center);

\draw [green, -{Stealth[length=3mm, width=2mm]}] ([shift=({2 cm, 0.1 cm})]A5.center) to[bend left=-1.5*\bendangle] ([shift=({-2 cm, 0.1 cm})]A6.center);
\draw [green, -{Stealth[length=3mm, width=2mm]}] ([shift=({-2 cm, 0.1 cm})]A6.center) to[bend right =-1.5*\bendangle] ([shift=({2 cm, 0.1 cm})]A5.center);

\draw [blue, -{Stealth[length=3mm, width=2mm]}] ([shift=({1.5 cm, -1.5 cm})]A4.center) to[bend left=-3*\bendangle] ([shift=({-1.5 cm, -1.5 cm})]A6.center);
\draw [blue, -{Stealth[length=3mm, width=2mm]}] ([shift=({-1.5 cm, -1.5 cm})]A6.center) to[bend left=3*\bendangle] ([shift=({1.5 cm, -1.5 cm})]A4.center);

\draw [red, -{Stealth[length=3mm, width=2mm]}] ([shift=({0.9 cm, -1.5 cm})]A1.center) to[bend left=-1.5*\bendangle] ([shift=({-1.8 cm, 1.3 cm})]A5.center);
\draw [red, -{Stealth[length=3mm, width=2mm]}] ([shift=({-1.8 cm, 1.3 cm})]A5.center) to[bend right=-1.5*\bendangle] ([shift=({0.9 cm, -1.5 cm})]A1.center);
\draw [green, -{Stealth[length=3mm, width=2mm]}] ([shift=({-1.5 cm, 1.6 cm})]A5.center) to[bend left=1.5*\bendangle] ([shift=({1.2 cm, -1.2 cm})]A1.center);

\draw [green, -{Stealth[length=3mm, width=2mm]}] ([shift=({-0.9 cm, -1.5 cm})]A3.center) to[bend left=1.5*\bendangle] ([shift=({1.8 cm, 1.3 cm})]A5.center);
\draw [green, -{Stealth[length=3mm, width=2mm]}] ([shift=({1.8 cm, 1.3 cm})]A5.center) to[bend right=1.5*\bendangle] ([shift=({-0.9 cm, -1.5 cm})]A3.center);
\draw [red, -{Stealth[length=3mm, width=2mm]}] ([shift=({1.5 cm, 1.6 cm})]A5.center) to[bend left=-1.5*\bendangle] ([shift=({-1.2 cm, -1.2 cm})]A3.center);

\draw [blue, -{Stealth[length=3mm, width=2mm]}] ([shift=({-1.9 cm, 1.1 cm})]A6.center) to[bend right=-1*\bendangle] ([shift=({1.4 cm, -0.9 cm})]A1.center);
\draw [green, -{Stealth[length=3mm, width=2mm]}] ([shift=({-1.7 cm, 1.3 cm})]A6.center) to[bend right=-1*\bendangle] ([shift=({1.6 cm, -0.6 cm})]A1.center);

\draw [red, -{Stealth[length=3mm, width=2mm]}] ([shift=({1.7 cm, 1.3 cm})]A4.center) to[bend left=-1*\bendangle] ([shift=({-1.6 cm, -0.6 cm})]A3.center);
\draw [blue, -{Stealth[length=3mm, width=2mm]}] ([shift=({1.9 cm, 1.1 cm})]A4.center) to[bend left=-1*\bendangle] ([shift=({-1.4 cm, -0.9 cm})]A3.center);

\draw [blue, -{Stealth[length=3mm, width=2mm]}] ([shift=({-0.9 cm, -1.5 cm})]A2.center) to[bend left=-1.5*\bendangle] ([shift=({1.2 cm, 1.8 cm})]A4.center);
\draw [blue, -{Stealth[length=3mm, width=2mm]}] ([shift=({1.2 cm, 1.8 cm})]A4.center) to[bend left=1.5*\bendangle] ([shift=({-0.9 cm, -1.5 cm})]A2.center);
\draw [red, -{Stealth[length=3mm, width=2mm]}] ([shift=({0.9 cm, 2 cm})]A4.center) to[bend left=1.5*\bendangle] ([shift=({-1.2 cm, -1.3 cm})]A2.center);

\draw [blue, -{Stealth[length=3mm, width=2mm]}] ([shift=({0.9 cm, -1.5 cm})]A2.center) to[bend right=-1.5*\bendangle] ([shift=({-1.2 cm, 1.8 cm})]A6.center);
\draw [blue, -{Stealth[length=3mm, width=2mm]}] ([shift=({-1.2 cm, 1.8 cm})]A6.center) to[bend right=1.5*\bendangle] ([shift=({0.9 cm, -1.5 cm})]A2.center);
\draw [green, -{Stealth[length=3mm, width=2mm]}] ([shift=({-0.9 cm, 2 cm})]A6.center) to[bend right=1.5*\bendangle] ([shift=({1.2 cm, -1.3 cm})]A2.center);

\end{tikzpicture}
\caption{A construction for 3 colors and a forbidden rainbow $S_{1,2}$ with nonempty sets $A$ and~$B$.}\label{fig:S_1,2_3AB}
\end{figure}
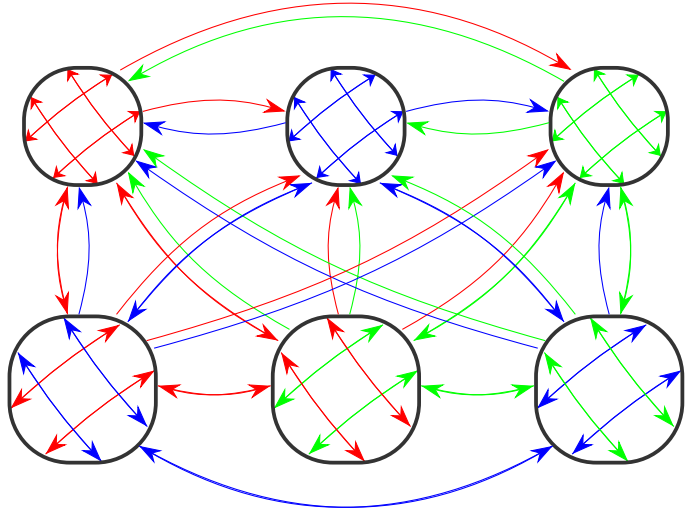

In order to prove the third bound in the first case of the theorem, consider $c$ satisfying $t_2 \leq c \leq t_4$ and take $x=c$ and $y=q-1$. 
From $(\ref{eq:S_p,q_jbound})$ we get
\[e(G_j) \leq \frac{1}{4c^3(q-1)}\big(2c(q-1)\alpha + (c+q-1)(p+q-1)\beta + (c^2 + (p-1)(q-1))\gamma\big)^2.\]
The assumption $c\leq t_4$ implies that $c^2 + (p-1)(q-1) \leq 2c(q-1)$, so
\begin{align*}
e(G_j) &\leq \frac{1}{4c^3(q-1)}\big(2c(q-1)(\alpha+\gamma) + (c+q-1)(p+q-1)\beta\big)^2 \\
&= \frac{1}{4c^3(q-1)}\big(2c(q-1)n - (c(q-p-1) - (q-1)(p+q-1))\beta\big)^2.
\end{align*}
Since $c \geq t_2$, the expression above is maximized at $\beta=0$ and gives
\[e(G_j) \leq \frac{q-1}{c}n^2,\]
which proves the third bound in the first case of the theorem.
This is the same bound as when forbidding a rainbow $S_{0,q}$, so it is achieved when $B = C = \emptyset$, while set $A$ is divided into $\binom{c}{q-1}$ equal-sized sets, one for each subset of $q-1$ colors assigned to the set, and we put edges in $G_i$, for $i\in [c]$, between any vertex from a set having color $i$ assigned to any other vertex. 
This is depicted for $q=3$ and $c=3$ in Figure~\ref{fig:S_0,3}.

Finally, consider the last remaining bound, where we have $c \geq t_3$ and $c \geq t_4$. 
From $(\ref{eq:S_p,q_jbound})$ for $x=c-p+1$ and $y=c-q+1$ we obtain
\begin{align*}
e(G_j) &\leq \frac{1}{4c^2(c\!-\!p\!+\!1)(c\!-\!q\!+\!1)}\big((c^2-(p\!-\!1)(q\!-\!1))(\alpha+\gamma) + (2c-p-q+2)(p+q-1)\beta\big)^2 \\
&= \frac{1}{4c^2(c\!-\!p\!+\!1)(c\!-\!q\!+\!1)}\big((c^2-(p\!-\!1)(q\!-\!1))n - ((c-p-q+1)^2-pq)\beta\big)^2.
\end{align*}
Since $c \geq t_3$, the above expression is maximized at $\beta = 0$ and gives
\[e(G_j) \leq \frac{(c^2-(p-1)(q-1))^2}{4c^2(c-p+1)(c-q+1)}.\]
This is achieved when $B = \emptyset$, the set $A$ is of size $\frac{(c-p+1)^2+(p-1)(q-p)}{2(c-p+1)(c-q+1)}n$ and is divided into $\binom{c}{q-1}$ equal-sized sets, one for each subset of $q-1$ colors assigned to the set, similarly, the set $C$ is of size $\frac{(c-q+1)^2-(q-1)(q-p)}{2(c-p+1)(c-q+1)}n$ and is divided into $\binom{c}{p-1}$ equal-sized sets, one for each subset of $p-1$ colors assigned to the set.
We put edges in $G_i$, for $i\in [c]$, from any vertex in $A$ having color $i$ assigned and any vertex in $C$ to any vertex in~$A$ and any vertex in $C$ having color $i$ assigned. 
This is depicted in Figure~\ref{fig:S_2,2_4+} for $c=4$ and a forbidden $S_{2,2}$ just for illustration, despite the fact that in this case it is not the extremal construction.
The sizes of sets $A$ and $C$ are non-negative if $c \geq t_4$, so it gives the last bound in both cases of the theorem.
\end{proof}

\begin{figure}[ht]
\centering
\begin{tikzpicture}[scale=0.5]
\def\radius{10}
\def\baseangle{40}
\def\shiftlvl{0.4}
\def\bendangle{10}
\def\bendangletwo{5}

\node[rectangle, draw=black!80, line width=0.5mm, minimum width=12cm, minimum height = 5cm, rounded corners=0.8cm] (A1) at (0,7) {};

\node[rectangle, draw=red!0, line width=0.5mm, minimum width=1.95cm, minimum height = 1.95cm, rounded corners=0.cm] (A21) at (150+0*\baseangle:\radius cm) {};
\pic [local bounding box=A2] at (150+0*\baseangle:\radius cm) {pic InnerGraph={red}{red}};

\node[rectangle, draw=red!0, line width=0.5mm, minimum width=1.95cm, minimum height = 1.95cm, rounded corners=0.cm] (A31) at (150-\baseangle:\radius cm) {};
\pic [local bounding box=A3] at (150-\baseangle:\radius cm) {pic InnerGraph={blue}{blue}};

\node[rectangle, draw=red!0, line width=0.5mm, minimum width=1.95cm, minimum height = 1.95cm, rounded corners=0.cm] (A41) at (150-2*\baseangle:\radius cm) {};
\pic [local bounding box=A4] at (150-2*\baseangle:\radius cm) {pic InnerGraph={green}{green}};

\node[rectangle, draw=red!0, line width=0.5mm, minimum width=1.95cm, minimum height = 1.95cm, rounded corners=0.cm] (A51) at (150-3*\baseangle:\radius cm) {};
\pic [local bounding box=A5] at (150-3*\baseangle:\radius cm) {pic InnerGraph={orange}{orange}};

\node[rectangle, draw=black!80, line width=0.5mm, minimum width=12cm, minimum height = 5cm, rounded corners=0.8cm] (A6) at (0,-7) {};

\node[rectangle, draw=red!0, line width=0.5mm, minimum width=1.95cm, minimum height = 1.95cm, rounded corners=0.cm] (A71) at (180+150+0*\baseangle:\radius cm) {};
\pic [local bounding box=A7] at (180+150+0*\baseangle:\radius cm) {pic InnerGraph={orange}{orange}};

\node[rectangle, draw=red!0, line width=0.5mm, minimum width=1.95cm, minimum height = 1.95cm, rounded corners=0.cm] (A81) at (180+150-\baseangle:\radius cm) {};
\pic [local bounding box=A8] at (180+150-\baseangle:\radius cm) {pic InnerGraph={green}{green}};

\node[rectangle, draw=red!0, line width=0.5mm, minimum width=1.95cm, minimum height = 1.95cm, rounded corners=0.cm] (A91) at (180+150-2*\baseangle:\radius cm) {};
\pic [local bounding box=A9] at (180+150-2*\baseangle:\radius cm) {pic InnerGraph={blue}{blue}};

\node[rectangle, draw=red!0, line width=0.5mm, minimum width=1.95cm, minimum height = 1.95cm, rounded corners=0.cm] (A101) at (180+150-3*\baseangle:\radius cm) {};
\pic [local bounding box=A10] at (180+150-3*\baseangle:\radius cm) {pic InnerGraph={red}{red}};

\draw [red, -{Stealth[length=3mm, width=2mm]}] ([shift=({1.7cm, 1.5cm})]A21.center) to[bend right=1.5*\bendangle] ([shift=({-1.5 cm, -1.6 cm})]A31.center);
\draw [blue, -{Stealth[length=3mm, width=2mm]}] ([shift=({-1.8cm, -1.3cm})]A31.center) to[bend right=1.5*\bendangle] ([shift=({1.3 cm, 1.8 cm})]A21.center);

\draw [blue, -{Stealth[length=3mm, width=2mm]}] (A31) to[bend left=1.5*\bendangle] (A41);
\draw [green, -{Stealth[length=3mm, width=2mm]}] (A41) ([shift=({-2 cm, 0 cm})]A41.center) to[bend left=1.5*\bendangle] ([shift=({2 cm, 0 cm})]A31.center);

\draw [green, -{Stealth[length=3mm, width=2mm]}] ([shift=({1.8cm, -1.3cm})]A41.center) to[bend left=1.5*\bendangle] ([shift=({-1.3 cm, 1.8 cm})]A51.center);
\draw [orange, -{Stealth[length=3mm, width=2mm]}] ([shift=({-1.6cm, 1.5cm})]A51.center) to[bend left=1.5*\bendangle] ([shift=({1.5 cm, -1.6 cm})]A41.center);

\draw [red, -{Stealth[length=3mm, width=2mm]}] ([shift=({2 cm, -0.8 cm})]A21.center) to[bend right=1.5*\bendangle] ([shift=({-2 cm, -0.8 cm})]A51.center);
\draw [orange, -{Stealth[length=3mm, width=2mm]}] ([shift=({-2 cm, -1.1 cm})]A51.center) to[bend left=1.5*\bendangle] ([shift=({2 cm, -1.1 cm})]A21.center);

\draw [red, -{Stealth[length=3mm, width=2mm]}] ([shift=({2.1 cm, 0.1 cm})]A21.center) to[bend right=1.5*\bendangle] ([shift=({-0.7 cm, -2 cm})]A41.center);
\draw [green, -{Stealth[length=3mm, width=2mm]}] ([shift=({-0.2 cm, -2 cm})]A41.center) to[bend left=1.5*\bendangle] ([shift=({2.1 cm, -0.2 cm})]A21.center);

\draw [orange, -{Stealth[length=3mm, width=2mm]}] ([shift=({-2 cm, 0.1 cm})]A51.center) to[bend left=1.5*\bendangle] ([shift=({0.7 cm, -2 cm})]A31.center);
\draw [blue, -{Stealth[length=3mm, width=2mm]}] ([shift=({0.2 cm, -2 cm})]A31.center) to[bend right=1.5*\bendangle] ([shift=({-2 cm, -0.2 cm})]A51.center);

\draw [green, -{Stealth[length=3mm, width=2mm]}] ([shift=({-1.6cm, -1.5cm})]A71.center) to[bend right=1.5*\bendangle] ([shift=({1.5 cm, 1.6 cm})]A81.center);
\draw [orange, -{Stealth[length=3mm, width=2mm]}] ([shift=({1.8cm, 1.3cm})]A81.center) to[bend right=1.5*\bendangle] ([shift=({-1.3 cm, -1.8 cm})]A71.center);

\draw [blue, -{Stealth[length=3mm, width=2mm]}] (A81) to[bend left=1.5*\bendangle] (A91);
\draw [green, -{Stealth[length=3mm, width=2mm]}] (A91) ([shift=({2 cm, 0 cm})]A91.center) to[bend left=1.5*\bendangle] ([shift=({-2 cm, 0 cm})]A81.center);

\draw [red, -{Stealth[length=3mm, width=2mm]}] ([shift=({-1.8cm, 1.3cm})]A91.center) to[bend left=1.5*\bendangle] ([shift=({1.3 cm, -1.8 cm})]A101.center);
\draw [blue, -{Stealth[length=3mm, width=2mm]}] ([shift=({1.6cm, -1.5cm})]A101.center) to[bend left=1.5*\bendangle] ([shift=({-1.5 cm, 1.6 cm})]A91.center);

\draw [red, -{Stealth[length=3mm, width=2mm]}] ([shift=({-2 cm, 0.8 cm})]A71.center) to[bend right=1.5*\bendangle] ([shift=({2 cm, 0.8 cm})]A101.center);
\draw [orange, -{Stealth[length=3mm, width=2mm]}] ([shift=({2 cm, 1.1 cm})]A101.center) to[bend left=1.5*\bendangle] ([shift=({-2 cm, 1.1 cm})]A71.center);

\draw [blue, -{Stealth[length=3mm, width=2mm]}] ([shift=({-2 cm, 0.2 cm})]A71.center) to[bend right=1.5*\bendangle] ([shift=({0.2 cm, 2 cm})]A91.center);
\draw [orange, -{Stealth[length=3mm, width=2mm]}] ([shift=({0.7 cm, 2 cm})]A91.center) to[bend left=1.5*\bendangle] ([shift=({-2 cm, -0.1 cm})]A71.center);

\draw [green, -{Stealth[length=3mm, width=2mm]}] ([shift=({2 cm, 0.2 cm})]A101.center) to[bend left=1.5*\bendangle] ([shift=({-0.2 cm, 2 cm})]A81.center);
\draw [red, -{Stealth[length=3mm, width=2mm]}] ([shift=({-0.7 cm, 2 cm})]A81.center) to[bend right=1.5*\bendangle] ([shift=({2 cm, -0.1 cm})]A101.center);

\draw [red, -{Stealth[length=3mm, width=2mm]}] ([shift=({0 cm, -2 cm})]A21.center) to[bend left=-1.5*\bendangle] ([shift=({0 cm, 2 cm})]A101.center);
\draw [blue, -{Stealth[length=3mm, width=2mm]}] ([shift=({0 cm, -2 cm})]A31.center) to[bend left=-1.5*\bendangle] ([shift=({0 cm, 2 cm})]A91.center);
\draw [green, -{Stealth[length=3mm, width=2mm]}] ([shift=({0 cm, -2 cm})]A41.center) to[bend left=1.5*\bendangle] ([shift=({0 cm, 2 cm})]A81.center);
\draw [orange, -{Stealth[length=3mm, width=2mm]}] ([shift=({0 cm, -2 cm})]A51.center) to[bend left=1.5*\bendangle] ([shift=({0 cm, 2 cm})]A71.center);

\draw [orange, -{Stealth[length=3mm, width=2mm]}] ([shift=({0 cm, -1 cm})]A6.east) to[bend left=-5*\bendangle] ([shift=({0 cm, 1 cm})]A1.east);
\draw [green, -{Stealth[length=3mm, width=2mm]}] ([shift=({0 cm, -0.3 cm})]A6.east) to[bend left=-5*\bendangle] ([shift=({0 cm, 0.3 cm})]A1.east);
\draw [blue, -{Stealth[length=3mm, width=2mm]}] ([shift=({0 cm, 0.3 cm})]A6.east) to[bend left=-5*\bendangle] ([shift=({0 cm, -0.3 cm})]A1.east);
\draw [red, -{Stealth[length=3mm, width=2mm]}] ([shift=({0 cm, 1 cm})]A6.east) to[bend left=-5*\bendangle] ([shift=({0 cm, -1 cm})]A1.east);

\end{tikzpicture}
\caption{A construction for 4 colors and a forbidden rainbow $S_{2,2}$ with nonempty sets $A$ and~$C$.}\label{fig:S_2,2_4+}
\end{figure}
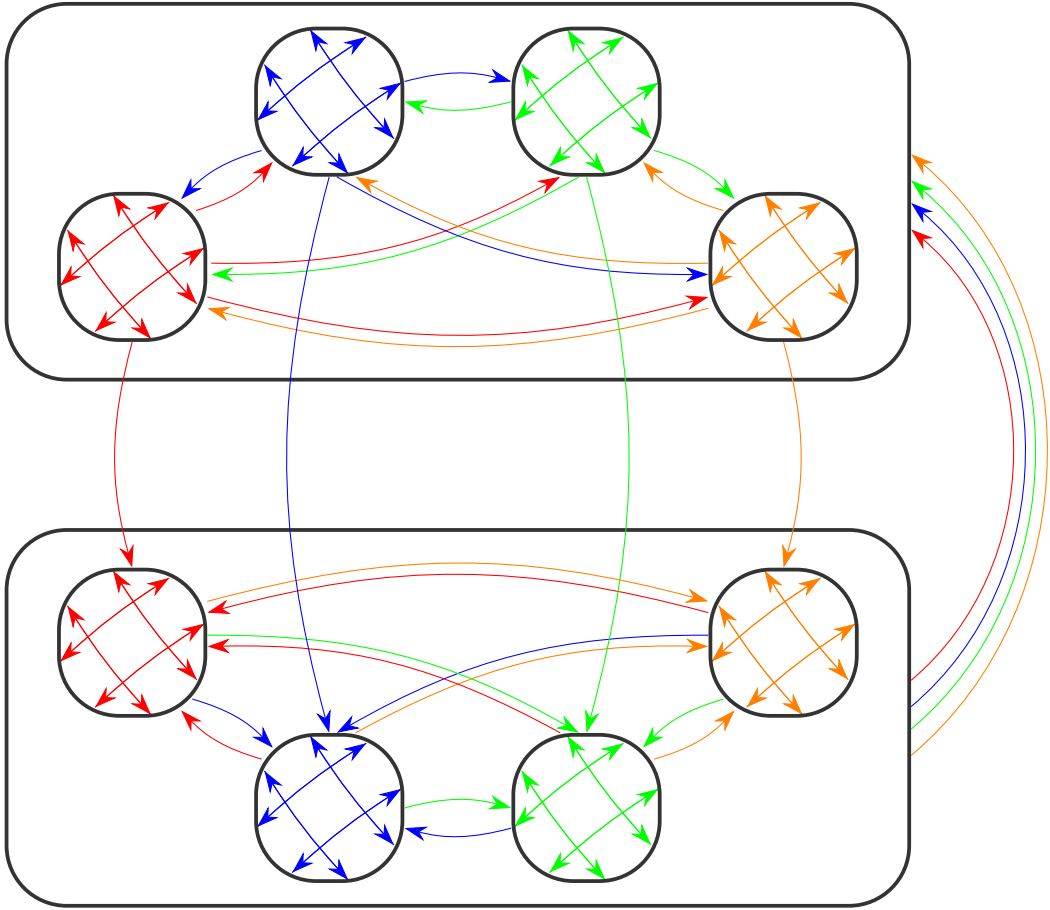

\section{Rainbow directed $S_{1,1}$}\label{sec:S_1,1}

In this section we consider $c \geq 2$ directed graphs $G_1, G_2, \ldots, G_c$ on a common vertex set and forbid a rainbow star $S_{1,1}$, which is a directed path of length 2. 
The following theorem provides the optimal bound for the sum of the number of edges in all graphs.

\begin{theorem}\label{thm:S_1,1_sum}
For  integers $c \geq 2$ and $n \geq 3$, every collection of directed graphs $G_1, \ldots, G_c$ on a common set of $n$ vertices containing no rainbow $S_{1,1}$ satisfies
\[\sum_{i=1}^c e(G_i) \leq 
\begin{cases} 
n^2-n & \ \mathrm{if}\ c \leq 3,\\
c\big\lfloor\frac{n^2}{4}\big\rfloor & \ \mathrm{if}\ c \geq 4.
\end{cases}
\]
Moreover, the above bounds are sharp.
\end{theorem}

\begin{proof}
The bound for $c\leq3$ is obtained when $G_1$ is a complete directed graph and all $G_i$ for $i \in [c]\setminus\{1\}$ are empty graphs.
While for $c\geq4$ the bound is achieved when each graph is the same balanced complete bipartite graph with all edges oriented in the same direction. This is depicted in Figure~\ref{fig:S_1,1}.

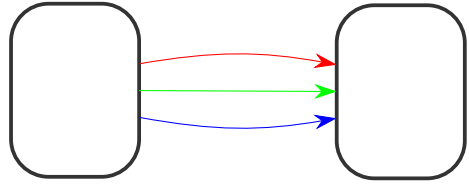
\begin{figure}[ht]
\centering
\begin{tikzpicture}
\def\radius{5}
\def\baseangle{30}
\def\bendangle{10}
\def\abovecorrection{-0.5}
\def\nodesizebigelipse{0.15}

\pic [local bounding box=A]                  {pic EmptyBox={0.7}{1.3}};
\pic [local bounding box=B, above right = \abovecorrection cm and 3cm of A]   {pic EmptyBox={0.7}{1.3}};

\draw [green, -{Stealth[length=3mm, width=2mm]}] (A) to (B); 

\begin{scope}[transform canvas={yshift=0.5em}]

\draw [red, -{Stealth[length=3mm, width=2mm]}] (A) to[bend left=\bendangle] (B); 

\end{scope}

\begin{scope}[transform canvas={yshift=-0.5em}]

\draw [blue, -{Stealth[length=3mm, width=2mm]}] (A) to[bend right=\bendangle] (B); 

\end{scope}

\end{tikzpicture}
\caption{The optimal construction for at least $4$ colors and a forbidden rainbow $S_{1,1}$.}\label{fig:S_1,1}
\end{figure}

We prove the upper bound by induction on $n$.  
For $n=3$ and $n=4$ it is easy to verify that the bound holds. 
Consider then a collection of directed graphs $G_1, \ldots, G_c$ on a common set $V$ of $n\geq5$ vertices containing no rainbow $S_{1,1}$ and assume that the bound holds for all collections on a smaller number of vertices.

For a subset $U \subset V$ and a vertex $v \in V$ by $e(U,v)$ we denote the total number of edges in all graphs $G_1, \ldots, G_c$ between the set $U$ and the vertex $v$. 
If $U$ consists of only one vertex~$u$, we write $e(u,v)$ instead of $e(\{u\},v)$ for brevity. 

Assume first that there are two vertices $u$ and $v$ such that there exists an edge $uv$ in at least two colors.
To avoid rainbow $S_{1,1}$, for each vertex $x \in V\setminus\{u,v\}$, there are no edges $xu$ nor $vx$, and if $ux \in E(G_i)$ and $xv\in E(G_j)$, then $i=j$. 
This means that $e(\{u,v\}, x) \leq c$. 
Moreover, if there is an edge $vu$ in some color, then between $\{u,v\}$ and $x$ we can have edges only in that color, so $e(\{u,v\}, x) \leq 2$. 
Consider those cases separately. 
If $e(u,v)\leq c$, then, together with the inductive assumption on $V \setminus \{u,v\}$, for $c\leq 3$ we obtain 
\[\sum_{i=1}^c e(G_i) \leq c + c(n-2) + (n-2)^2-(n-2) = n^2-n-(3-c)(n-1)-(n-3)\leq n^2 - n,\]
while for $c\geq4$ we have
\[\sum_{i=1}^c e(G_i) \leq c + c(n-2) + c\left\lfloor\frac{(n-2)^2}{4}\right\rfloor = c\left\lfloor\frac{n^2}{4}\right\rfloor.\]
On the other hand, if $e(u,v)>c$, then for $c\leq 3$ we obtain 
\[\sum_{i=1}^c e(G_i) \leq 2c + 2(n-2) + (n-2)^2-(n-2) = n^2 - n -2(n-4)-2(3-c) \leq n^2-n,\]
while for $c\geq4$ we have
\[\sum_{i=1}^c e(G_i) \leq 2c + 2(n-2) + c\left\lfloor\frac{(n-2)^2}{4}\right\rfloor = c\left\lfloor\frac{n^2}{4}\right\rfloor -(c-4)(n-3)-2(n-4) \leq c\left\lfloor\frac{n^2}{4}\right\rfloor.\]

Therefore, we are left with the case when between any pair of vertices there are at most two edges (at most one edge in each direction). But then $\sum_{i=1}^c e(G_i) \leq 2\binom{n}{2} = n^2-n$, which gives the correct bound for $c \leq 3$ and is smaller than $c\left\lfloor\frac{n^2}{4}\right\rfloor$ for $c \geq 4$, as desired.
\end{proof}

Since the extremal construction for $c\geq4$ in Theorem~\ref{thm:S_1,1_sum} has the same number of edges in each graph, this theorem immediately implies the optimal bound for $\min_{1\leq i \leq c} e(G_i)$ if $c\geq4$. We will show that for $c \leq 3$ the same bound holds. 

\begin{theorem}\label{thm:S_1,1_2+}
For any integers $c \geq 2$ and $n \geq 4$, every collection of directed graphs $G_1, \ldots, G_c$ on a common set of $n$ vertices containing no rainbow $S_{1,1}$ satisfies
\[\min_{1\leq i \leq c} e(G_i) \leq \left\lfloor\frac{n^2}{4}\right\rfloor.\]
Moreover, this bound is sharp.
\end{theorem}

\begin{proof}
The bound is achieved when each graph is the same balanced complete bipartite graph with all edges oriented in the same direction (see Figure~\ref{fig:S_1,1}).

Note that it is enough to consider $c=2$, because if the theorem is true for 2 colors, then it also holds for any larger number of colors. 
We proceed by induction on $n$. 
For $n=4$ it is easy to verify that the theorem holds (note that for $n=3$ it is not true as one can have a directed triangle oriented clockwise in $G_1$ and anticlockwise in $G_2$). 
Consider then two directed graphs $G_1$, $G_2$ on a common
set $V$ of $n \geq 5$ vertices containing no rainbow $S_{1,1}$ and assume that the theorem holds for all pairs of graphs on a smaller number of vertices.

Assume first that there exists a vertex $v \in V$ incident to at most $2$ edges in each of the graphs. 
Then, from the induction assumption on $V\setminus\{v\}$ we obtain
\[\min_{1\leq i \leq 2} e(G_i) \leq 2 + \left\lfloor\frac{(n-1)^2}{4}\right\rfloor =  \left\lfloor\frac{n^2-2n+9}{4}\right\rfloor \leq  \left\lfloor\frac{n^2}{4}\right\rfloor.\]
Therefore, we may assume that no such vertex exists in $V$.
In particular, every vertex has outdegree or indegree at least $2$ in at least one of the graphs.

We split the vertex set $V$ into disjoint sets based on colors and directions of incident edges, similarly as in Section~\ref{sec:S_p,q}. 
Let $\alpha$ and $\gamma$ be the number of vertices that have outdegree~0, respectively indegree~0, in both graphs.
For $i \in [2]$, let $\beta_i$ be the number of remaining vertices that are incident only to edges in $G_i$.
The remaining vertices in $V$ can be divided into $4$ sets depending whether outdegree or indegree is large and in which graph. 
For $i\in[2]$, let $\delta_i^+$ be the number of vertices having outdegree at least $2$ in $G_i$ and $\delta_i^-$ the number of vertices having indegree at least $2$ in $G_i$. 
Let $M$ be the set of pairs of vertices $u,v$ such that $uv \in E(G_1)$ and $vu \in E(G_2)$. 
Note that every vertex appears in at most one pair in $M$ as otherwise it has no other incident edges which contradicts the condition proven in the previous paragraph. 
Moreover, all $\delta_1^+ + \delta_1^- + \delta_2^+ + \delta_2^-$ vertices must appear in $M$.
Thus, we can set $m = |M| = \frac{1}{2}(\delta_1^+ + \delta_1^- + \delta_2^+ + \delta_2^-)$.

Note that all edges in $G_1$, except those in $M$, are going from $\gamma + \beta_1 + \delta_1^+$ vertices and to $\alpha + \beta_1 + \delta_1^-$ vertices. 
From symmetry, we may assume without loss of generality that 
\[\beta_1 + \frac{1}{2}\delta_1^+ + \frac{1}{2}\delta_1^- \leq \beta_2 + \frac{1}{2}\delta_2^+ + \frac{1}{2}\delta_2^-,\] 
we obtain
\begin{align*}
e(G_1) &\leq (\gamma + \beta_1 + \delta_1^+)(\alpha + \beta_1 + \delta_1^-) + m \\
&\leq \left\lfloor\frac{(\gamma + \beta_1 + \delta_1^+ + \alpha + \beta_1 + \delta_1^-)^2}{4}\right\rfloor + m\\
&\leq \left\lfloor\frac{(\gamma + \alpha + \beta_1 + \beta_2 + \frac{1}{2}\delta_1^+ + \frac{1}{2}\delta_1^- + \frac{1}{2}\delta_2^+ + \frac{1}{2}\delta_2^-)^2}{4}\right\rfloor + m\\
&= \left\lfloor\frac{(n-m)^2}{4}\right\rfloor + m\\
&= \left\lfloor\frac{n^2 - m(2n-m-4)}{4}\right\rfloor\\
&\leq \left\lfloor\frac{n^2}{4}\right\rfloor,
\end{align*}
because $m \leq \frac{1}{2}n$ and $n\geq4$, which concludes the proof.
\end{proof}

\subsection*{Acknowledgements}

We are grateful to the organizers of the 14th Eml\'ekt\'abla Workshop where the authors worked on the problems presented in this paper.

\end{document}